\documentclass[12pt]{amsart}
\usepackage{fullpage, times}
\usepackage{amscd,amssymb,amsmath,amsthm, color}
\usepackage{graphicx}

\newtheorem{Proposition}{Proposition}

\newtheorem{Lemma}{Lemma}
\newtheorem{Theorem}{Theorem}
\newtheorem*{Theorem*}{Theorem}
\newtheorem{Corollary}{Corollary}

\newtheorem{defn}{Definition}

\newcommand{\proj}{\mathbb{P}}

\newcommand{\Z}{\mathbb{Z}}

\newcommand{\com}{\mathbb{C}}
\newcommand{\Q}{\mathbb{Q}}

\newcommand{\MM}{\mathsf{M}}

\newcommand{\bx}{\mathbf {x}}
\newcommand{\by}{\mathbf {y}}
\newcommand{\bz}{\mathbf {z}}
\newcommand{\bw}{\mathbf {w}}
\newcommand{\bmu}{{\boldsymbol{\mu}}}
\newcommand{\bnu}{{\boldsymbol{\nu}}}
\newcommand{\balpha}{{\boldsymbol{\alpha}}}

\newcommand{\uuu}{{\mathfrak{1}}}
\newcommand{\ppp}{{\mathfrak{p}}}

\newcommand{\lv}{\left |}

\newcommand{\lang}{\left\langle}
\newcommand{\rang}{\right\rangle}

\newcommand{\combinatfactor}{\mathfrak{z}}
\newcommand{\zz}{{\combinatfactor}}

\newcommand{\cF}{\mathcal{F}}
\newcommand{\vac}{v_\emptyset}
\DeclareMathOperator{\Aut}{Aut}

\def\rootQ{\sqrt{Q}}
\def\rootQ2{\sqrt{Q_2}}
\def\sl2{\mathfrak{sl}_2}

\def\H{\mathcal{H}}

\begin{document}

\title{A Fock Space approach to Severi Degrees of Hirzebruch Surfaces}
\author{Y. Cooper}
\date{August 2017}
\maketitle

\begin{abstract}
The classical Severi degree counts the number of algebraic curves of fixed genus and class passing through some general points in a surface.  In this paper we study Severi degrees as well as several types of Gromov--Witten invariants of the Hirzebruch surfaces $F_k$, and the relationship between these numbers.  To each Hirzebruch surface $F_k$ we associate an operator $\mathsf{M}_{F_k} \in \H[\proj^1]$ acting on the Fock space $\cF[\proj^1]$.  Generating functions for each of the curve-counting theories we study here on $F_k$ can be expressed in terms of the exponential of the single operator $\mathsf{M}_{F_k}$, and counts on $\proj^2$ can be expressed in terms of the exponential of $\mathsf{M}_{F_1}$.  Several previous results can be recovered in this framework, including the recursion of Caporaso and Harris for enumerative curve counting on $\proj^2$, the generalization by Vakil to $F_k$, and the relationship of Abramovich--Bertram between the enumerative curve counts on $F_0$ and $F_2$.   We prove an analog of Abramovich-Bertram for $F_1$ and $F_3$.  We also obtain two differential equations satisfied by generating functions of relative Gromov--Witten invariants on $F_k$.  One of these recovers the differential equation of Getzler and Vakil.

\end{abstract}

\section{Introduction}
Two important pieces of work that contributed to the modern renaissance of enumerative algebraic geometry were the papers of Kontsevich and Caporaso--Harris, in which they gave complete formulas for counting curves in $\proj^2$, first in genus 0 \cite{KM}, and subsequently in all genera \cite{CH}.  The counts of genus 0 curves on $\proj^2$ appear as structure coefficients in the quantum cohomology ring of $\proj^2$, which has a rich algebraic structure.  However as Caporaso and Harris remark at the beginning of their seminal paper:

\begin{quotation}
One aspect of the work of Kontsevich and Ruan--Tian is that they relate these numbers to the coefficients in the structure equations of an algebraic object, the quantum cohomology ring. It would be very interesting to see if any similar interpretation could be placed on the degrees of Severi varieties in positive genus. We don't at present know of any algebraic structure that generates these numbers. \cite{CH}
\end{quotation}

\noindent In this paper we describe an algebraic structure that generates the curve counts in all genera for several different curve counting theories on the surfaces $\proj^2$ and the Hirzebruch surfaces $F_k$.  It is not a generalization of the quantum cohomology ring.

\subsection{Acknowledgements}
I would like to thank R. Pandharipande, J. Harris, A. Patel, A. Pixton, P. Etingof, D. Ranganathan, Q. Chen, and R. Vakil for many conversations and ideas.

\subsection{Background and notation}
\subsubsection{Severi theory}

This paper builds on the results of \cite{CP}.  We refer to that paper for a more detailed description of the set up.

Let $S$ be a nonsingular projective surface.  The moduli space of stable maps
$$\overline{M}_{g,n}^\bullet(S,\beta)$$
from genus $g$, $n$-pointed curves to $S$ representing the class $\beta\in H_2(S,\mathbb{Z})$ has virtual dimension
$$\text{dim}_{\com}\ [\overline{M}^\bullet_{g,n}(S,\beta)]^{vir} \ =\ g-1 +n+ \int_\beta c_1(S).$$
The superscript $\bullet$ indicates the domain is possibly disconnected, but with no connected components collapsed to points of $S$.
Let
$$\text{ev}_i: \overline{M}^\bullet_{g,n}(S,\beta) \rightarrow S$$
be the evaluation at the $i^{th}$ marked point.
We refer the reader to \cite{FP,KM} for an introduction to stable maps and Gromov--Witten theory.

A Gromov--Witten analogue of the Severi degree is defined by the following
construction. Let
$$n= \int_\beta c_1(S)+g-1$$
be the virtual dimension of the unpointed space $\overline{M}^\bullet_{g}(S,\beta)$.
Let
$$N^{\bullet}_{g,\beta} = \int_{[\overline{M}^\bullet_{g,n}(S,\beta)]^{vir}} \prod_{i=1}^{n} \text{ev}_i^*(\ppp),$$
where $\ppp\in H^4(S,\mathbb{Z})$ is the point class.  If $n<0$, then $N_{g,\beta}^\bullet$ vanishes by definition.

For an arbitrary surface $S$, the Gromov--Witten invariant $N^{\bullet}_{g,\beta}$ may be completely unrelated to the classical Severi degree. Indeed, for Enriques surfaces, the Gromov--Witten invariants are often fractional, and, for $K3$ surfaces, the Gromov--Witten invariants vanish altogether. For the surfaces $\proj^2$ and $\proj^1\times \proj^1$, $N^{\bullet}_{g,\beta}$ coincides with the (disconnected) classical Severi degree \cite{CP}.  For the Hirzebruch surfaces $F_k$, when $k>0$ the numbers $N^{\bullet}_{g,\beta}$ are not enumerative, and in Section (\ref{AbramBert}) we discuss the relationship between $N^{\bullet}_{g,\beta}$ and enumerative curve counts on $F_k$.  There are several additional ways to define curve counting theories on $F_k$ and in this paper we discuss some of them.

The first we consider is relative Gromov--Witten invariants, as defined by Li in \cite{L}.  Consider a surface $S$ and a smooth divisor (possibly disconnected) $D$ on it.  As in Gromov--Witten theory, one can consider maps of genus $g$ curves to $S$ whose image is the homology class $\beta$.  However now we also fix a cohomology weighted partition $\eta$ which specifies how the curve meets $D$.

In this paper, we will have $S = F_k$ and $D = C \bigcup E$.  In this case the relative divisor $D$ is disconnected, and we will use the following notation for clarity.  Let $\mu_C [\uuu] +  \nu_C[\ppp]$  denote the relative condition on $C$ and $\mu_E[\uuu] + \nu_E[\ppp]$ the relative condition on $E$.  

Let
$$
\overline{M}^\bullet_{g,n+m}(\mu_C, \nu_C, \mu_E, \nu_E)(S,\beta)
$$
denote the moduli space of relative stable maps of genus $g$, $n+m$-pointed curves mapping to $S$ in the class $\beta$ and meeting $C$ (resp. $E$) with multiplicity $\mu_C [\uuu]+\nu_C[\ppp]$ along $C$ and $\mu_E[\uuu]+\nu_E[\ppp]$ along $E$.  The relative conditions are {\it labeled}.

As before, $\int_\beta c_1(S)+g-1$ is the virtual dimension of the unpointed space $\overline{M}^\bullet_{g}(S,\beta)$, so we impose $$n = \int_\beta c_1(S)+g-1 - (|\mu_C+\mu_E| -\ell(\mu_C+\mu_E) + |\nu_C+\nu_E|)$$ general point conditions on the map.  We also have $m=\ell(\mu+\nu)$ marked points which map to the relative conditions.

We define the relative Gromov--Witten invariants of $F_k$ relative to $C$ and $E$ as

\begin{align*}
N^{\bullet}_{g,\beta} & (\mu_C, \nu_C, \mu_E, \nu_E)  \\ 
& = \frac{1}{|Aut(\mu_C)||Aut(\nu_C)||Aut(\mu_E)||Aut(\nu_E)|}\int_{[\overline{M}^\bullet_{g,n}(\mu_C, \nu_C, \mu_E, \nu_E)(S,\beta)]^{vir}} \prod_{i=1}^{n} \text{ev}_i^*(\ppp).
\end{align*}

A special case of relative Gromov--Witten invariants that will be interesting for us is the case of the most general conditions possible are placed on the relative divisor.  Namely $\mu_C$ and $\mu_E$ are partitions all of whose parts are 1, and $\nu_C$ and $\nu_E$ are empty.  We will call these transverse Gromov--Witten invariants, and will denote them as follows.

\begin{align*}
\hat{N}^{\bullet}_{g,\beta} = \frac{1}{|Aut(\mu_C)||Aut(\mu_E)|}\int_{[\overline{M}^\bullet_{g,n}((1,...,1), \emptyset, (1,...,1), \emptyset)(S,\beta)]^{vir}} \prod_{i=1}^{n} \text{ev}_i^*(\ppp).
\end{align*}

Lastly, one can consider enumerative curve counts on $F_k$, as defined in \cite{V}.  In the remainder of this paper we will discuss formulas for all of the curve counting theories discussed in this introduction and study the relationships between them.

\subsubsection{Fock Space}
To start, we write the cohomology of $\proj^1$ as  the standard direct sum
$$H^*(\proj^1,\Q) = \Q \cdot \uuu \ \oplus\ \Q \cdot \ppp\ $$
where $\uuu$ and $\ppp$ are the unit and point classes respectively.

The Lie algebra $\H [ \proj^1]$ is generated by the operators $1,$ $\{ \alpha_k [\uuu]\},$ and $\{ \alpha_k [\ppp]\}$, where $k \in \Z \setminus \{0\}$. The Lie bracket is given by
\begin{eqnarray} \label{Hp1commutation}
\left[\alpha_k[\uuu],\alpha_l[\ppp]\right] &= &k \, \delta_{k+l,0}\,
\end{eqnarray}
with all other commutators vanishing.

The Fock space $\cF[\proj^1]$ is freely generated over $\Q$ by the elements
$$\alpha_{-k_1}[\uuu]...\alpha_{-k_n}[\uuu] \alpha_{-\ell_1} [\ppp] ... \alpha_{-\ell_m} [\ppp] \vac, \ \ k_i, \ell_j \in\Z_{>0},$$
where by the commutation relations, the order of $\alpha_{-k_1}[\uuu]...\alpha_{-k_n}[\uuu] \alpha_{-\ell_1} [\ppp] ... \alpha_{-\ell_m} [\ppp]$ does not matter.

$\H [\proj^1]$ acts on $\cF[\proj^1]$ in a way analogous to the action of $\H$ on $\cF$.  The creation operators
$$\alpha_{-k} [\uuu] \ \mathrm{and} \, \alpha_{-k} [\ppp], k \in \Z_{>0},$$
act via
$$ \alpha_{-k} [\uuu] \left(\prod_{i,j} \alpha_{-k_i} [\uuu] \alpha_{-\ell_j} [\ppp] \vac \right) = \alpha_{-k} [\uuu] \prod_{i,j} \alpha_{-k_i} [\uuu] \alpha_{-\ell_j} [\ppp] \vac,$$
and similarly for $\alpha_{-\ell} [\ppp].$
The annihilation operators
$$\alpha_{k} [\uuu] \ \mathrm{and } \, \alpha_{k} [\ppp],\ \  k\in\Z_{>0},$$ kill the vacuum
$$
\alpha_k [\uuu] ( \vac )=0 \ \mathrm{and} \, \alpha_k [\ppp] ( \vac )=0, \quad k>0 \,,
$$
and their action on any element $\prod_{i,j} \alpha_{-k_i} [\uuu] \alpha_{-\ell_j} [\ppp] \vac$ is determined by the commutation relations (\ref{Hp1commutation}).

A natural basis of $\cF[\proj^1]$ is given by the vectors
\begin{equation*}
  \label{basis}
  \lv \mu,\nu \rang = \frac{1}{\zz(\mu)\zz(\nu)}
 \, \prod_{i=1}^{\ell (\mu )} \alpha_{-\mu_i}[\uuu]
\prod_{j=1}^{\ell (\nu )} \alpha_{-\nu_j}[\ppp]
 \, \vac \,
\end{equation*}
indexed by all pairs of partitions $\mu$ and $\nu$ (of possibly different sizes), where $\zz(\mu)$ denotes the combinatorial factor
$$\zz(\mu)= |\Aut(\mu)|\cdot \prod_{i=1}^{\ell(\mu)} \mu_i \ \ .$$
An inner product is defined by
\begin{equation*}
  \label{inner_prod}
  \lang \mu,\nu | \mu',\nu' \rang =
\frac{u^{-\ell(\mu)}}{\zz(\mu)}
\frac{u^{-\ell(\nu)}}{\zz(\nu)} \delta_{\mu\nu'}\delta_{\nu\mu'}
\,.
\end{equation*}

We define a new operator $\MM_{F_k}$ on the Fock space $\cF[\proj^1]$ by the following formula,
$$\MM_{F_k}(u,Q_1, Q_2) = \sum_{i>0} \alpha_{-i}[\ppp] \alpha_i[\ppp]
+ Q_1^k Q_2  \sum_{|\mu|-k =|\nu| \geq 0} u^{({\ell(\mu)}-1)} \alpha_{-\mu}[\uuu]\alpha_\nu[\uuu]\ \ .$$
The second sum is over all pairs of partitions $\mu$ and $\nu$ whose size differ by $k$, and
\begin{eqnarray*}
\alpha_{-\mu}[\uuu]&=& \frac{1}{|\Aut(\mu)|}
\prod_{i=1}^{\ell(\mu)} \alpha_{-\mu_i}[\uuu]\ , \\
\alpha_{\nu}[\uuu]&=& \frac{1}{|\Aut(\nu)|}
\prod_{i=1}^{\ell(\nu)} \alpha_{\nu_i}[\uuu]\ . \\
\end{eqnarray*}
The variable $u$ encodes the genus, and the variables $Q_1$ and $Q_2$ encode the curve class.

\subsection{Statement of results}

\subsubsection{Gromov--Witten invariants of $F_k$}
The Picard group of $F_k$ is generated by the fiber class $F$ and the class of the exceptional divisor $E$.  The intersection products are 
$$F \cdot F = 0$$

$$F \cdot E = 1$$

$$E \cdot E = -k.$$

Let the variables $Q_1$ and $Q_2$ correspond to the curve classes $F$ and $E$ respectively.  

The partition function for the Gromov--Witten invariants of $F_k$ is

$$\mathsf{Z}^{F_k}(u, Q_1, Q_2, t) =  1+\sum_{g\in \mathbb{Z}} u^{g-1} \sum_{(d_1,d_2)} \  N_{g,(d_1,d_2)}^{\bullet} \frac{t^n}{n!} \ Q_1^{d_1} Q_2^{d_2}
$$
where the second sum is over all non-negative $d_i$ satisfying $(d_1,d_2)\neq (0,0)$ and $n = g-1+2d_1+(2-k)d_2$.

We define vectors $v$ and $w_k$ as follows.

\begin{equation}
v = exp(\alpha_{-1}[\uuu]) v_{\emptyset}
\end{equation}

and

\begin{equation}
w_k = exp \left( \sum_{\mu, \nu}  \left( \int_{ [ \overline{M}_h ( F_k(a,b) \backslash \mu,\nu)]^{vir}} \hspace{-0.8in} 1 \hspace{0.7in} \right) \ Q_1^s Q_2^t u^{h+\ell(\mu)+\ell(\nu) -1} \ \ \alpha_{-\mu}[\ppp] \alpha_{-\nu}[\uuu] \right) v_\emptyset,
\end{equation}
where $\mu,\nu$ are partitions and $\overline{M}_h ( F_k(a,b) \backslash \mu,\nu)$ is a moduli space of relative stable maps parameterizing stable maps of genus $h$ curves mapping to $F_k$ in the class $a F + b E$ relative the divisor $C$, with relative conditions given by $\mu [\uuu] + \nu [\ppp]$.  

In Section \ref{pfmain} we will prove that the following formula computes the generating function for Gromov--Witten invariants of $F_k$ using the degeneration formula of Li \cite{L}.
 
\begin{Theorem} \label{HirzebruchGW}
For every Hirzebruch surface $F_k$,
$$\mathsf{Z^{F_k}} (u, Q_1, Q_2,  t) = \big\langle\ v \ | \  \exp\big(t\mathsf{M}_{F_k}(u,Q_1, Q_2)\big) \ | \ w_k\ \big\rangle.$$
\end{Theorem}

For $k \leq 3$, we can compute $w_k$ explicitly.  The case of $F_0$ was already studied in \cite{CP}, where $w_0$ was computed to be $Q_1 exp \left(\alpha_{-1}[\ppp] \right)$.  

\begin{Proposition}\label{ws}

For $F_1$,
$$w_1 = exp\left(  Q_1\alpha_{-1}[\ppp] + \frac{Q_2}{u} \right)v_{\emptyset}. $$
On $F_2$,
$$w_2 = exp\Bigl( Q_1 \alpha_{-1}[\uuu] + Q_1 Q_2 \alpha_{-1}[\uuu] \Bigr) v_{\emptyset}.$$
On $F_3$,
$$w_3 = exp\left( Q_1 \alpha_{-1}[\uuu] + \frac{1}{2} Q_1^2 Q_2 u \alpha_{-1}^2[\uuu] +  Q_1^2 Q_2 \alpha_{-2}[\uuu]  + \sum_{a > 0} \frac{(-1)^{a-1}}{a^2} Q_1^a Q_2^a \alpha_{-a}[\ppp] \right) v_{\emptyset}.$$

\end{Proposition}

The formula of Abramovich-Bertram \cite{AB} comparing the Gromov--Witten invariants on $F_2$ to enumerative curve counts on $F_2$ can be recovered by comparing $w_0$ and $w_2$.  This will be discussed in Section (\ref{AbramBert}).

\subsubsection{Other computations for $F_k$}

We can also compute generating functions of relative and transverse GW invariants, and find an interesting connection here to the enumerative geometry of $F_k$ which was studied by Block and Goettsche using tropical methods.

Now we consider relative Gromov--Witten invariants of $F_k$.  In order to assemble them in a generating function, we introduce $\bx, \by,\bz,\bw$.  We use the following notation.  If we write the partition $\mu$ as $\mu = 1^{m_1},2^{m_2},...$ then let $\bx^\bmu $ denote
$$
\bx^\bmu = x_1^{m_1} x_2^{m_2}...
$$

Let 
$$|\bmu_C, \bnu_C \rangle = |\mu_C, \nu_C \rangle \bx^{\bmu_C} \by^{\bnu_C}$$
and
$$|\bmu_E, \bnu_E \rangle = |\mu_E, \nu_E \rangle \bz^{\bmu_E} \bw^{\bnu_E}.$$

The partition function for the relative Gromov--Witten invariants of $F_k$ is

$$\mathsf{Z}^{F_k}(u, Q_1, Q_2, t, \bx,\by,\bz,\bw) =  1+\sum_{g\in \mathbb{Z}} u^{g-1} \sum_{(d_1,d_2)} \  N_{g,(d_1,d_2)}^{\bullet} (\mu_C, \nu_C, \mu_E, \nu_E) \frac{t^n}{n!} \ Q_1^{d_1} Q_2^{d_2} \bx^{\mu_C} \by^{\nu_C} \bz^{\mu_E} \bw^{\nu_E} 
$$
where the second sum is over all non-negative $d_i$ satisfying $(d_1,d_2)\neq (0,0)$ and $n = g-1+2d_1+(2-k)d_2 - (|\mu|+|\nu| - \ell(\nu))$.

\begin{Theorem} \label{relHirzebruchGW}
For every Hirzebruch surface $F_k$,
$$\mathsf{Z^{F_k}} (u, Q_1, Q_2,  t, \bx,\by,\bz,\bw) = \sum_{ \bmu_C, \bmu_E, \bnu_C, \bnu_E}\big\langle\ \bmu_C, \bnu_C | \  \exp\big(t\mathsf{M}_{F_k}(u,Q_1, Q_2)\big) \ | \ Q_1^{|\mu_E+\nu_E|} \bmu_E, \bnu_E\ \big\rangle.$$
\end{Theorem}

Transverse Gromov--Witten invariants as we defined in the previous section are a special case of relative Gromov--Witten invariants, and we obtain a formula for them as a result of Theorem \ref{relHirzebruchGW}.  The partition function for the transverse Gromov--Witten invariants of $F_k$ is

$$\mathsf{\hat{Z}}^{F_k}(u, Q_1, Q_2, t) =  1+\sum_{g\in \mathbb{Z}} u^{g-1} \sum_{(d_1,d_2)} \  \hat{N}_{g,(d_1,d_2)}^{\bullet} \frac{t^n}{n!} \ Q_1^{d_1} Q_2^{d_2}
$$
where the second sum is over all non-negative $d_i$ satisfying $(d_1,d_2)\neq (0,0)$ and $n = g-1+2d_1+(2-k)d_2$.

By restricting the choice of partitions $\mu,\nu$, we obtain the following corollary as a special case of Theorem \ref{relHirzebruchGW}.
\begin{Corollary} \label{logHirzebruchGW}
For every Hirzebruch surface $F_k$,
$$\mathsf{\hat{Z}^{F_k}} (u, Q_1, Q_2,  t) = \big\langle\ v \ | \  \exp\big(t\mathsf{M}_{F_k}(u,Q_1, Q_2)\big) \ | \ w_0\ \big\rangle.$$
\end{Corollary}

While the Gromov--Witten invariants of $F_0$ are enumerative, in general for $F_k$ they are not.  However, it turns out that these transverse GW invariants are enumerative for all $F_k$.  Following Block--Goettsche \cite{BG} we define $\widetilde{N}_{g,(d_1,d_2)}^{\bullet}$ to be the number of (possibly disconnected) curves of genus $g$ in $| L(d_1 F + d_2 E)|$ passing through $n = g-1+2a+(2-k)b$ general points in $F_k$ which do not contain $E$ as a component.  We collect these enumerative curve counts in a generating function as well.

$$
\mathsf{\widetilde{Z}}^{F_k}(u, Q_1, Q_2, t) =1+\sum_{g\in \mathbb{Z}} u^{g-1} \sum_{(d_1,d_2)} \  \widetilde{N}_{g,(d_1,d_2)}^{\bullet} \frac{t^{n}}{n!} \ Q_1^{d_1} Q_2^{d_2}
$$
where the second sum is over all non-negative $d_i$ satisfying $(d_1,d_2)\neq (0,0)$.

Comparing Corollary \ref{logHirzebruchGW} to the following formula proved by Block and Goettsche gives a relationship between the Gromov--Witten invariants and enumerative geometry of the Hirzebruch surfaces $F_k$. 

\begin{Theorem*}[Block--Goettsche, \cite{BG}]
$\mathsf{\widetilde{Z}}^{F_k} (u, Q_1, Q_2, t) = \big\langle\ v \ | \  Q_1^{|\cdot|} \exp \big(t\mathsf{M}_{F_k}(u,Q_1,Q_2)\big) \ | w_0\ \big\rangle.$
\end{Theorem*}

Thus we conclude that

\begin{Corollary}
The transverse Gromov--Witten invariants of the Hirzebruch surface $F_k$ are enumerative.  Namely,
$$
\mathsf{\hat{Z}}^{F_k} (u, Q_1, Q_2, t) = \mathsf{\widetilde{Z}}^{F_k} (u, Q_1, Q_2, t).
$$
\end{Corollary}

\subsubsection{Formulas for $\proj^2$}
Let the variable $Q$ correspond to the hyperplane class.  The partition function for Severi degrees of $\proj^2$ is
$$
\mathsf{Z}^{\proj^2} (u, Q, t) =1+\sum_{g\in \mathbb{Z}} u^{g-1} \sum_{d} \  N_{g,d}^{\bullet} \frac{t^{n}}{n!} \ Q^d
$$
where $n = 3d+g-1$.

\begin{Theorem} \label{P2}
${\mathsf Z}^{\proj^2} (u, Q, t)= \big\langle\ v \ | \ \exp\big(t\mathsf{M}_{F_1}(u,Q)\big) \ | v_{\emptyset}\ \big\rangle.$
\end{Theorem}

In \cite{CP} a method for computing the generating function for connected Gromov--Witten invariants of $\proj^2$ is given.  Theorem \ref{P2} is a more direct formula for computing the Gromov--Witten invariants of $\proj^2$.

As with the Hirzebruch surfaces, we can also study the relative Gromov--Witten invariants of $\proj^2$ with relative conditions $\mu[\uuu]+\nu[\ppp]$ imposed along a line $L$.  The proof of Theorem \ref{P2} also yields the following formula for relative Gromov--Witten invariants on $\proj^2$.
\begin{Corollary} \label{relP2}
${\mathsf Z}^{\proj^2}(u, Q, t, \bx, \by)   = \big\langle\ \bmu, \bnu \ | \ \exp\big(t\mathsf{M}_{F_1}(u,Q)\big) \ | v_{\emptyset}\ \big\rangle.$
\end{Corollary}

\section{Proof of Theorem \ref{HirzebruchGW}}\label{pfmain}
\subsection{Overview}
We prove Theorem \ref{HirzebruchGW} via the degeneration formula for relative Gromov--Witten invariants \cite{IP,LR,L}.
Consider the Gromov--Witten invariant $N_{g,(d_1,d_2)}^{\bullet}$ counting genus $g$ curves in $F_k$ in the class $d_1 F + d_2 E$, passing through $$n=g-1+2d_1+(2-k)d_2$$ points.

We will compute $N_{g,(d_1, d_2)}^{\bullet}$ by applying the degeneration formula to the following degeneration.  Let $X = F_k \times \Delta$, where $\Delta$ is the disc.  In the fiber $X_0 \simeq F_k$, let $C_k$ be a curve in the class $E + kF$ (classically $C_k$ is called a co-directrix in $F_k$).  Consider the three-fold $Bl_{C_k} X$.  The special fiber is the union of two copies of $F_k$, with the divisor $C_k$ in the original $F_k$ glued to the exceptional divisor $E$ in the second $F_k$. 

We iterate this construction $n+1$ times, and obtain a chain of $n+2$ surfaces, each isomorphic to $F_k$.  This is the degeneration of $F_k$ we will use, and we distribute the original $n$ point conditions by placing one on each of the middle $n$ components.

\begin{figure}[h] \label{DegenerationPic}
\centering
\includegraphics[width=6 in]{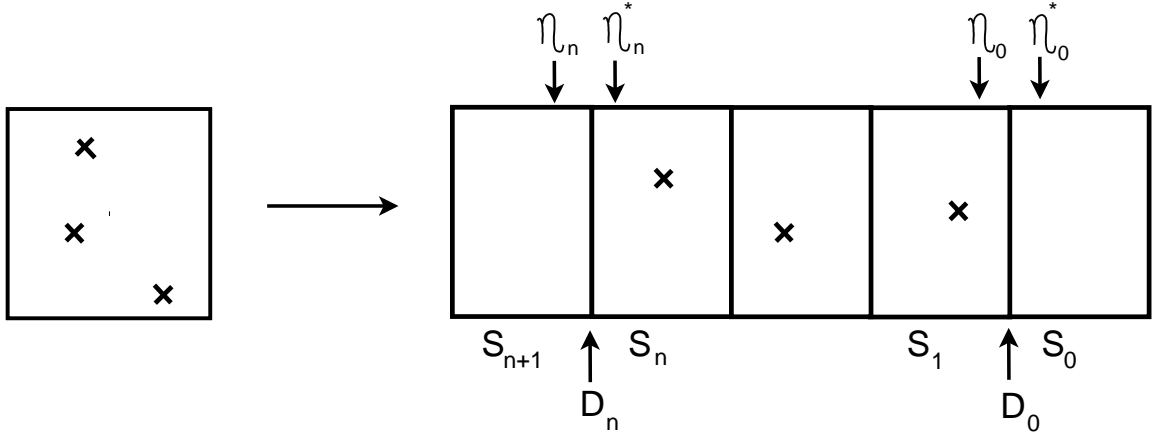}
\caption{The first and last components carry no point conditions, all the rest have one point condition each.  Note that $(1,0)$ is the class of $F$ (horizontal line) and $(0,1)$ represents the class $E + kF$.}
\end{figure}

We will refer to the $n+2$ components of the degeneration as
$S_{0}, \ldots , S_{n+1}$
and the $n+1$ relative divisors as $D_0,\ldots,D_n$.
The matrix $\mathsf{M}_{F_k}(u, Q_1, Q_2, t) $ arises from explicit calculations
on components of this degeneration.

\subsection{Stable relative maps}
A moduli space of stable relative maps is defined for each component $S_i$ of the
above degeneration.
Relative conditions along the divisors $D_i$ are specified
by partitions weighted by
 the cohomology of $\proj^1$.
For $S_i$ where
 $1 \leq i \leq n$, let
$$M^i = \overline{M}^{\bullet}_{g_i,1}(D_i \backslash\, F_k / D_{i-1} , (d_1, d_2^{(i)}), \eta^*_{i}, \eta_{i-1}, \Gamma^i)$$
be the moduli space of stable relative maps of graph type $\Gamma^i$ to the component $S_i$ in the class $d_1 F + d_2^{(i)} E$ satisfying relative conditions $\eta_{i-1}$ along the divisor $D_{i-1}$ and  $\eta^*_{i}$ along $D_{i}$.

The graph type $\Gamma^i $ fixes the topology of the map.  Each vertex of $\Gamma^i$
corresponds to a component of the domain curve and is labeled with the genus of that component.
For each relative condition on that domain curve the vertex is given a half-edge labeled
with the corresponding relative condition. The unique marked point is assigned
to a single component of $\Gamma^i$ (which satisfies the incidence condition).

The components $S_0$ and $S_{n+1}$ play a special role. Following the above
conventions, let
$$M^0 = \overline{M}^{\bullet}_{g_0,0}(D_0 \backslash F_k, (d_1, d_2^0), \eta^*_0, \Gamma^0)$$
and let
$$M^{n+1} = \overline{M}^{\bullet}_{g_{n+1},0}( F_k / D_{n}, (d_1, d_2^{n+1}), \eta_n, \Gamma^{n+1})\ .$$
For all the above moduli spaces $M^i$, we will view the relative markings  on the domain of the
map
as {\em ordered}.

\subsection{Partition notation}
We take all our partitions to be ordered partitions.
\begin{defn}
Let $\rho$ be a partition of $d$ and let $\rho(k)$ be the number of parts of size $k$ in $\rho$, so $d = \sum_{k = 1}^{\infty} \rho(k) k$.
\end{defn}

Let $\rho = \rho_1 + \ldots + \rho_m$ and $\lambda = \lambda_1 + \ldots + \lambda_n$ be two partitions, and $d = |\rho| + |\lambda|$.  We say
$$\rho [\uuu] + \lambda [\ppp] = \rho_1 [\uuu] + \ldots + \rho_m[\uuu] + \lambda_1 [\ppp]
+ \ldots + \lambda_n [\ppp]$$
is a cohomology weighted partition of $d$, weighted by the cohomology of $\proj^1$.

Let $\cup$ denote concatenation of partitions,
$$\rho \cup \lambda = \rho_1 + \ldots + \rho_m + \lambda_1 + \ldots + \lambda_n.$$

\begin{defn}
Let $\eta = \rho [\uuu] + \lambda [\ppp]$ be a partition weighted by the cohomology of $\proj^1$.  Let
$$m(\eta)= \prod_i \rho_i \prod_j \lambda_j,\ \ \
\text{\em Aut}(\eta) = \text{\em Aut}(\rho) \times \text{\em Aut}(\lambda),\ \ \
\eta^* = \lambda [\uuu] + \rho [\ppp].$$
\end{defn}
For example, if $\eta = 2 [\uuu] + [\uuu] + [\uuu] + 3 [\ppp]$, $m(\eta) = 6$ and $|Aut(\eta)| = 2$.

\subsection{Degeneration} \label{poneponeDegen}

By the degeneration formula of \cite{IP,LR,L}, 

\begin{align} \label{degenformula}
N^\bullet_{g,(d_1,d_2)}
& = \sum_{d^i_2,\eta_i,\Gamma^i}
\left(\int_{[M^{n+1}]} 1 \right) \frac{m(\eta_{n})}{ |\text{Aut}(\eta_n)|} \prod_{i = 1}^{n} \left[ \left(\int_{[M^i]} \text{ev}_1^*(\ppp) \right) \frac{m(\eta_{i-1})}{ |\text{Aut}(\eta_{i-1})|}  \right]
\left(\int_{[M^{0}]} 1 \right)\ .
\end{align}
The sum is over all degree splittings $$d_2^0 + \cdots + d_2^{n+1}=d_2,$$
relative conditions $\eta_0, \ldots, \eta_n$, and compatible
graph types $\Gamma^0,\ldots \Gamma^{n+1}$ which connect to form a genus $g$ curve.
The relative conditions $\eta_i^*$ are set by Definition 2.
On the right side, $[M^i]$ denotes the virtual fundamental class of the moduli
space $M^i$.

Equivalently, we can write the partition function of the Severi degrees of the surface
$F_k$ as
\begin{multline} \label{RelativeEqn}
 \mathsf {Z}^{F_k}(u, Q_1, Q_2, t) = 1 + \sum_{g,d_1,d_2} Q_1^{d_1} Q_2^{d_2} u^{g-1}\frac{t^n}{n!} \sum_{d_2^{(i)},\eta_i,\Gamma^i}
\left(\int_{[M^{n+1}]} 1 \right) \frac{m(\eta_{n})}{ |\text{Aut}(\eta_n)|} \\
 \times \prod_{i = 1}^{n} \left[ \left(\int_{[M^i]} \text{ev}_1^*(\ppp) \right)
\frac{m(\eta_{i-1})}{ |\text{Aut}(\eta_{i-1})|}  \right]
\left(\int_{[M^{0}]} 1 \right).
\end{multline}
In the above formula, $n=g-1+2d_1+(2-k)d_2$ as usual.

\subsection{Geometry of the tube components}  \label{tube}
We begin by analyzing the components  $S_i$ for $1 \leq i \leq n$.  Let
$$\eta_{i-1} = \rho[\uuu] + \lambda[\ppp], \ \ \  \eta_i^* = \rho'[\uuu] + \lambda' [\ppp].$$
We consider a single genus $h$ connected component $R$ of the domain curve of a map to $S_i$.  Let
$$\sigma[\uuu] + \tau[\ppp] \ \ \text{and} \ \ \sigma'[\uuu] + \tau'[\ppp]$$
be the relative conditions imposed on $R$ along $D_{i-1}$ and $D_{i}$ respectively.
Let $f_*[R] = \beta = sF + t E$.  In the relative geometry, no component of $R$ is allowed to fall into the relative divisor.  So the class $\beta$ must satisfy $s \geq kt$.  

As $M^i$ is a moduli space of 1-pointed curves, there are two cases: either the marked point lies on $R$ or $R$ is unpointed.

Consider first the case where $R$ does not carry a marked point. Then,
$$\mathrm{dim}_\com \ \overline{M}_{h,0}(F_k, \beta)  =  2s + (2-k)t + h - 1\ .$$
The relative conditions impose
\begin{align*}
\sum (\sigma_i-1) + \sum (\sigma'_{i'}-1) + \sum \tau_j  + \sum \tau_{j'} &= (sF_k+tE)\cdot C_k + (sF_k+tE) \cdot E - \ell(\sigma) - \ell(\sigma') \\
& = 2s - kt -\ell(\sigma) - \ell(\sigma')
\end{align*}
conditions, where $\sigma_i$ are the parts of the partition $\sigma$, and so on.
After equating dim $\overline{M}_{h,0}(F_k, \beta)$ with the number of relative conditions imposed, we obtain
$$2t+ h + \ell(\sigma) + \ell(\sigma') = 1. $$

The two solutions (up to exchanging $\sigma$ and $\sigma'$) are

$$h = 0,\  t = 0,\ \ell(\sigma) = 1, \ \ell(\sigma') = 0.$$
$$h = 1,\  t = 0,\ \ell(\sigma) = 0, \ \ell(\sigma') = 0.$$

Only the first is geometrically possible, so there is a unique configuration allowed, consisting of a rational curve mapping with unconstrained degree $m$ to the line in the class $(1,0)$.  We see $\sigma' = \tau = \emptyset$ and $\sigma = \tau'$ must be a partition with a single part.  The map is ramified totally over $D_{i-1}$ and $D_i$, and the image in the class $(1,0)$ is determined by the fixed condition $\tau_1' [\ppp]$ .  The moduli space of such maps is isomorphic as a stack to $B (\Z/m)$.

We consider next the case where $R$ carries the marked point.  Then
$$\mathrm{dim}_\com \ \overline{M}_{h,1}(F_k, \beta) = 2s + (2-k)t + h.$$
As in the previous case, the relative conditions impose
$$\sum (\sigma_i-1) + \sum (\sigma'_{i'}-1) + \sum \tau_j + \sum \tau_{j'} = 2s - kt - \ell(\sigma) - \ell(\sigma')$$
conditions on such a map.  Setting their difference equal to 2, the degree of $\text{ev}_1^*(\ppp)$, we obtain
$$2t + h + \ell(\sigma) + \ell(\sigma') = 2.$$
Again, the terms on the left hand side are all nonnegative.  The arithmetically allowed solutions are
\begin{eqnarray*} 
h= 0,\ t=0,\ \ell(\sigma) = \ell(\sigma') = 1,\\
h = 0,\ t = 1, \ \ell(\sigma) = \ell(\sigma') = 0,\\
h=0, t = 0, \ell(\sigma) = 2, \ell(\sigma') = 0,\\
h=1,t=0,\ell(\sigma)=1,\ell(\sigma')=0,\\
\mathrm{and} \ \ h=2,t=0,\ell(\sigma)=0,\ell(\sigma')=0
\end{eqnarray*}
with $s$ and $\ell(\tau)$ unconstrained.

However, due to geometric constraints, only the following two types of configurations can appear:
\begin{eqnarray*} \label{TypeA}
\mathrm{Type \ A:} & & \ \ h= 0,\ t=0,\ \ell(\sigma) = \ell(\sigma') = 1, s \neq 0\\
 \label{TypeBC}
\mathrm{Type \ B:} & &\ \ h = 0,\ t = 1, \ \ell(\sigma) = \ell(\sigma') = 0,\ s \geq k. \\
\end{eqnarray*}

If $R$ is a component of Type A, then $\sigma = \sigma' = |\sigma|$ is a partition with only one part and $R$ is a 1-pointed rational curve mapping with degree $|\sigma|$ to a line in the class $(1,0)$ totally ramified over the two relative divisors $D_{i-1}$ and $D_i$. The moduli space $M$ of such maps is isomorphic to $\proj^1$.  Since $R$ has a marked point and the map has two ramification points, there are no automorphisms of this map and $\int_{[M]} ev_1^*(\ppp) = 1$.

If $R$ is a component of Type B, then $R$ is a 1-pointed rational curve mapping to $F_k$ in the class $(s,1)$, subject to the relative conditions $\sigma_i, \sigma'_{i'}, \tau_j, \tau_{j'}$.  By writing explicit equations that cut out a general curve in class $(s,1)$, one can compute $\int_{[M]} ev_1^*(\ppp) = 1$.

In conclusion, if the partitions $\eta_{i-1}$ and $\eta_i$ are such that the domain curve has $k$ components and the $k-1$ unmarked components map to $F_k$ with degrees $m_1,...,m_{k-1}$ repsectively, then
$$\int_{[M^i]} ev_1^*(\ppp) = \frac{1}{m_1 \cdot ... \cdot m_{k-1}}.$$

\subsection{Geometry of the cap $S_{n+1}$} \label{section:leftcap}
We will now analyze the integrals appearing in the degeneration formula \eqref{RelativeEqn}.

On the component $S_{n+1}$, let the relative condition be
$$\eta_n = \rho[\uuu] + \lambda [\ppp],$$
where $\rho$ and $\lambda$ are partitions satisfying $| \rho | + | \lambda | = d_1.$

Let $R$ be a component of the domain curve of a map to $S_{n+1}$ parameterized
by $M^{n+1}$, and let
$$\sigma[\uuu] + \tau[\ppp]$$
be the relative condition imposed on $R$.  Suppose the genus of $R$ is $h$ and $f_*[R] = \beta = (s,t)$.  The dimension of the space of such maps is
$$\mathrm{dim}_\com \ \overline{M}_{h,0}(F_k, \beta)  = \int_{\beta} c_1(T_{F_k}) + h-1 = 2s + (2-k)t+ h - 1.$$
Meanwhile the number of conditions imposed on the map by the relative conditions is
$$\sum (\sigma_i - 1) + \sum \tau_j = (sF + tE)(E) - \ell(\sigma) = s- kt - \ell(\sigma).$$
After setting the dimension to equal the number of conditions,  we obtain
$$s+2t + h + \ell(\sigma) = 1.$$

Each term on the left hand side is nonnegative, so $t$ must be zero.  Since $\beta$ is nonzero, $s$ cannot also vanish.  Then the only possible solution is
$$s=1 \ \mathrm{and} \ t=h=\ell(\sigma)=0.$$
Therefore $R$ must be a genus 0 curve mapping with degree 1 onto a line in the class $(1,0)$ with the relative condition $[\ppp]$.

We conclude the integral over $M^{n+1}$ vanishes unless
$$d_2^{n+1} = 0, \  \eta_n = [\ppp] + \ldots + [\ppp],$$
and $\Gamma^{n+1}$ is a graph on $d_1$ vertices with a half-edge at each vertex.
If the above conditions are satisfied, the moduli space $M^{n+1}$ consists of
a single point which parameterizes a map of $d_1$ disconnected rational curves to $F_k$, each with a fixed condition of multiplicity 1, mapping with degree 1 to the unique curve in the class $(1,0)$ passing through that fixed relative condition.
As such a map has no automorphisms, we find
$$\int_{[M^{n+1}]} 1 =1. $$

\subsection{Proof of main theorem}

We are now ready to prove Theorem \ref{HirzebruchGW}.
\begin{proof}

We will show that the matrix product
\begin{equation}\label{gfft}
 \big\langle\ v \ | \ \exp\big(t\mathsf{M}_{F_k}(u,Q_1, Q_2)\big) \ | \ w_k\ \big\rangle.
 \end{equation}
is the partition function for Severi degrees of $F_k$.

We prove this by showing that the coefficient of $u^{g-1} Q_1^{d_1} Q_2^{d_2}$ in equation (\ref{gfft}) is equal to $N^\bullet_{g,(d_1,d_2)} \frac{t^n}{n!}$ as computed by Li's degeneration formula.  Fix $g, d_1, d_2$ and hence $n$.  Take the degeneration of $F_k$ outlined at the beginning of this section.  Our analysis of the geometry of the tube components shows that the choice of the sequence of partitions $\eta_i$ in fact determines the sequence $d_2^{(i)}$.  We rearrange the sum in equation (\ref{degenformula}).

\begin{align} \label{degenformulaii}
N^\bullet_{g,(d_1,d_2)}
& = \sum_{\eta_i} \sum_{\Gamma^i}
\left(\int_{[M^{n+1}]} 1 \right) \frac{m(\eta_{n})}{ |\text{Aut}(\eta_n)|} \prod_{i = 1}^{n} \left[ \left(\int_{[M^i]} \text{ev}_1^*(\ppp) \right) \frac{m(\eta_{i-1})}{ |\text{Aut}(\eta_{i-1})|}  \right]
\left(\int_{[M^{0}]} 1 \right)\ .
\end{align}
The first sum is now over all relative conditions $\eta_0, \ldots, \eta_n$, such that the implied degrees $d_2^{(i)}$ satisfy $$d_2^0 + ... + d_2^{n+1}=d_2.$$  The second sum is over all sequences of graphs $\Gamma^i$ which are consistent with the tube geometries studied in Section \ref{tube}.  As $n, d_1,$ and $d_2$ have already been fixed, restricting to such $\Gamma_i$ automatically ensures that the curve formed in the degenerate surface by connecting those in each component connect to form a genus $g$ curve.

To match equation (\ref{degenformulaii}) with the coefficient of $u^{g-1} Q_1^{d_1} Q_2^{d_2}$ in equation (\ref{gfft}), match the two summands of $\mathsf{M_{F_k}}$ with the two configuration types $A$ and $B$ of Section \ref{tube} respectively.  The vector $w_k$ is the sum of all possible relative conditions $\eta_0$ with coefficient $\int_{[M_0]} 1$ and the appropriate combinatorial factor.  We leave the bookkeeping to the reader. 

\end{proof}

\subsection{Geometry of cap $S_0$}

Now we study the geometry of the cap $S_0$ and obtain some constraints on the summands in the vector $w_k$.  We also compute $w_k$ explicitly for the Hirzebruch surface $F_k$ when $k$ is small.  

Consider the component $S_0$.  Let the relative condition be
$$\eta^*_0 = \rho[\uuu] + \lambda [\ppp].$$
where $\rho$ and $\lambda$ are partitions satisfying $|\rho|+|\lambda| = d_1$.

Let $R$ be a component of the domain curve of a map to $S_0$ parameterized by $M^0$, and let
$$\sigma[\uuu] + \tau[\ppp]$$
be the relative condition imposed on $R$.  Suppose the genus of $R$ is $h$ and $f_*[R] = \beta = (s,t)$.  The dimension of the space of such maps with no relative conditions imposed is
$$\mathrm{dim}_\com \ \overline{M}_{h,0}(F_k, \beta)  =  h-1  +  \int_{(s,t)} c_1(T_{F_k})=  h-1 + 2s + (2-k)t.$$
Meanwhile the number of conditions imposed on the map by the relative conditions is
$$\sum (\sigma_i - 1) + \sum \tau_j = (sF + tE) \cdot (C) - \ell(\sigma) = s-\ell(\sigma).$$
After setting the dimension to equal the number of conditions,  we obtain
\begin{equation} \label{CapNumerics}
 h + s + \ell(\sigma) + (2-k) t = 1.
 \end{equation}

The following lemma gives a second constraint on the vector $w_k$.  

\begin{Lemma} \label{CapInequality}
For $k= 2m$ and $2m+1$, the moduli space $\overline{M}_h(F_{k},(aF + bE) \backslash C, \mu,\nu)$ is nonempty only if $a \geq m b$.
\end{Lemma}

\begin{proof}
The even and odd Hirzebruch surfaces are deformation equivalent.  So by the deformation invariance of relative Gromov--Witten theory, 

$$\overline{M}_h(F_{2m},(aF + bE) \backslash C, \mu,\nu) \simeq \overline{M}_h(F_{0},(aF + b(E-mF)) \backslash C+mF, \mu,\nu)$$
and
$$\overline{M}_h(F_{2m+1},(aF + bE) \backslash C, \mu,\nu) \simeq \overline{M}_h(F_{1},(aF + b(E-mF)) \backslash C+mF, \mu,\nu).$$

In both cases, for the moduli space on the right to be nonempty, the curve class must be effective and hence we must have $a-mb \geq 0$.  
\end{proof}

In \cite{CP} we computed $w_0$, and here we are able to compute $w_k$ explicitly for $k=1,2,3$.

\subsubsection{$F_1$}

On $F_1$, the dimension constraint (\ref{CapNumerics}) becomes 
$$h + s + t + \ell(\sigma) = 1.$$

The numerically allowed solutions are

\begin{enumerate}
\item $h=1, s= 0, t =0, \ell(\sigma) = 0$

\item $h= 0, s=1, t=0, \ell(\sigma)= 0$

\item $h=0, s= 0, t= 1, \ell(\sigma) = 0$

\item $h= 0, s=0, t=0, \ell(\sigma)= 1.$
\end{enumerate}

Now we consider which of the above are geometrically possible.  Neither (1) or (4) is, because the curve class vanishes.  The other two cases can appear.

(2)  In this case, $\nu={1}$ and the moduli space on the cap is $$\overline{M}_0(F_1, F \setminus \nu ) \simeq \overline{M}_0(\proj^1, 1) \simeq pt.$$ 

So 
$$\int_{\overline{M}_0(F_1, F \setminus \nu ) }1 = 1.$$

(3)  In this case, $\nu = \emptyset$ and the moduli space on the cap is $$\overline{M}_0(F_1, E \setminus \nu) \simeq \overline{M}_0(\proj^1,1) \simeq pt$$ 

so
$$\int_{\overline{M}_0(F_1, E)} 1 = 1.$$

We conclude $$w_1 = exp\left(Q_1 \alpha_{-1}[\ppp] + \frac{Q_2}{u}\right)v_{\emptyset}. $$

\subsubsection{$F_2$}
For the surface $F_2$, the condition (\ref{CapNumerics}) specializes to 

$$
h+s+\ell(\sigma) = 1.
$$

The numerically allowed solutions which also satisfy Corollary \ref{CapInequality} are:

\begin{enumerate}

\item $h= 0, s=1,t=0, \ell(\sigma)= 0.$ 

\item $h= 0, s=1,t=1, \ell(\sigma)= 0.$ 

\end{enumerate}

Now we analyze these possibilities.

(1) In this case, $\mu = \emptyset , \nu = 1$, and 

$$\int_{\overline{M}_0(F_2, F \backslash C, \{\mu, \nu \} )} 1 =  1.$$  

(2)  Again, $\mu = \emptyset $ and $ \nu= 1$.  We evaluate the integral by transporting the problem to $F_0$.  

 $$\int_{\overline{M}_0(F_2, F + E \backslash C, \{\mu, \nu \} ) } 1 =  \int_{\overline{M}_0(F_0, E \backslash C+F, \{\mu, \nu \} )}  = 1.$$  

We conclude that $$w_2 = exp \Bigl( Q_1 \alpha_{-1}[\uuu] +Q_1 Q_2 \alpha_{-1}[\uuu] \Bigr) v_{\emptyset}$$

\subsubsection{$F_3$}

On $F_3$, the dimension constraint (\ref{CapNumerics}) becomes 

\begin{equation} \label{F3ineq}
h + s - t + \ell(\sigma) = 1.
\end{equation}

There are now infinitely many numerically allowed solutions and we begin by ruling out many of these solutions for geometric reasons. 

First, we apply Lemma (\ref{CapInequality}) to conclude that only solutions where $s \geq t$ must be considered.  Hence if we group terms, equation (\ref{F3ineq}) becomes

$$h + ( s - t)  + \ell(\sigma) = 1$$

and each of the three summands is nonnegative.  Let $r = s-t$.  There are then three cases.  

\begin{enumerate}
\item $h=1, r=0, \ell(\sigma) = 0$

\item $h= 0, r=1, \ell(\sigma)= 0$

\item $h= 0, r=0, \ell(\sigma)= 1.$
\end{enumerate}

Using the same idea as in Lemma (\ref{CapInequality}), we translate the problem to a relative Gromov--Witten calculation on $F_1$.  
\begin{equation} \label{F3sub}
 \overline{M}_{h,0}(F_3, sF+tE \backslash C, \{ \mu, \nu \}) \simeq \overline{M}_{h,0}(F_1, rF + tE \backslash C+F, \{ \mu, \nu \} )
\end{equation}

(1)  In this case, $\ell(\sigma) = 0$ so $\mu = 0$.  Equation \ref{F3sub} becomes

$$ \overline{M}_{h,0}(F_3, tF+tE \backslash C, \{\nu \} ) \simeq \overline{M}_{h,0}(F_1, tE \backslash C+F, \{ \nu \} )$$

But the moduli space on the right hand side is empty, as the curve class is a multiple of $E$ and there's no way for a curve in that class to satisfy fixed relative conditions on $C+F$.  So there is no contribution from this case.

(2)  In this case equation (\ref{F3sub}) becomes
$$ \overline{M}_{h,0}(F_3, t F+ (t-1) E \backslash C, \{ \nu \}) \simeq \overline{M}_{h,0}(F_1, F + tE \backslash C+F, \{ \nu \}).$$

Any curve on $F_1$ in the class $F+tE$ must have at least one component isomorphic to $E$ as soon as $t \geq 2$.  But since all the relative conditions are fixed, there is no way for the curve to meet the relative curve $C+F$ as required.  Hence the moduli space on the right hand side is empty if $t \geq 2$. 

That leaves two cases, either $t=0$ or $t=1$.  

In the case $t=0$, $\nu = 1$.  
The moduli space $\overline{M}_{h,0}(F_3, F \backslash C, \{ \nu \} ) \simeq pt$, hence we get a contribution of  $Q_1 \alpha_{-1}[\uuu]$ .

In the case $t=1$, the contribution is nonzero only if $\nu$ is either $(1,1)$ or $(2)$.

If $\nu = (1,1)$, $\int_{\overline{M}_0 ( F_1(2,1 \backslash \nu = \{1,1\} ) } 1 = 1$ so the contribution is $\frac{1}{2} Q_1^2 Q_2 u \alpha_{-1}^2[\uuu] $

If $\nu =(2)$, $\int_{\overline{M}_0 ( F_1(2,1) \backslash \nu = \{2\})} 1 = 2$ so the contribution is $ Q_1^2 Q_2 \alpha_{-2}[\uuu] $

(3)  Now equation (\ref{F3sub}) becomes

$$ \overline{M}_{h,0}(F_3, tF+tE \backslash C ,\{ \mu, \nu \} ) \simeq \overline{M}_{h,0}(F_1, tE \backslash C+F, \{ \mu, \nu\} )$$
where $\mu$ is a partition with one part.

The moduli space on the right hand side is empty if $\nu$ has any parts.  However, it can be nonempty if $\nu=0$ and $\mu = t$ has a single part.  In this case, this space consists of maps that are degree $m$ map from $\proj^1$ to $E$, totally ramified where $E$ intersects the relative divisor $C+F$.  

The integral we must compute is 
$$\int_{[\overline{M}_{h,0}(F_1, tE \backslash \mu = t )]^{vir}} 1 = \int_{\overline{M}_{h,0} ( \proj^1, t \backslash pt, \{ \mu = t \}) } c_{top}( H^1(f^* (O(-1))) ).$$

This integral was computed by Bryan and Pandharipande in \cite{BP} to be $ \frac{(-1)^{t-1}}{t^2}.$ 

We conclude that the contribution in this case is

$$\sum_{a > 0} \frac{(-1)^{a-1}}{a^2} Q_1^a Q_2^a \alpha_{-a}[\ppp].$$

We conclude
$$w_3 = exp\left( Q_1 \alpha_{-1}[\uuu] + \frac{1}{2} Q_1^2 Q_2 u \alpha_{-1}^2[\uuu] +  Q_1^2 Q_2 \alpha_{-2}[\uuu]  + \sum_{a > 0} \frac{(-1)^{a-1}}{a^2} Q_1^a Q_2^a \alpha_{-a}[\ppp] \right) v_{\emptyset}.$$

\subsection{Proof of Theorem \ref{relHirzebruchGW}}

The proof of Theorem \ref{HirzebruchGW}, with some modification, yields a proof also of Theorem \ref{relHirzebruchGW}.

\begin{proof}

The relative Gromov--Witten invariants of $F_k$ can also be computed via the degeneration formula.  To compute $N_{g,(d_1,d_2)}^{\bullet} (\mu_C, \nu_C, \mu_E, \nu_E)$, we again degenerate to the normal cone of $E$ multiple times.  This time we create a singular surface with $n$ components, all isomorphic to $F_k$, where $n = g-1+2d_1+(2-k)d_2 - (|\mu|+|\nu| - \ell(\nu))$.  We place one point condition in each component.  The degeneration formula now gives the following expression.

\begin{align}
N_{g,(d_1,d_2)}^{\bullet} (\mu_C, \nu_C, \mu_E, \nu_E) & = \sum_{d^i_2,\eta_i,\Gamma^i}
\frac{m(\eta_{n})}{ |\text{Aut}(\eta_n)|} \prod_{i = 1}^{n} \left[ \left(\int_{[M^i]} \text{ev}_1^*(\ppp) \right) \frac{m(\eta_{i-1})}{ |\text{Aut}(\eta_{i-1})|}  \right]
 \ .
\end{align}
The sum is over all degree splittings $$d_2^0 + \cdot + d_2^{n+1}=d_2,$$
relative conditions $\eta_0, \ldots, \eta_n$, and compatible
graph types $\Gamma^1,\ldots \Gamma^{n}$ which connect to form a genus $g$ curve.
The relative conditions $\eta_0 = \mu_E[\uuu] + \nu_E[\ppp] $ and $\eta_n = \nu_C[\uuu]+\mu_C[\ppp] $ are fixed by the problem.
On the right side, $[M^i]$ denotes the virtual fundamental class of the moduli
space $M^i$.

Now we have no cap components, only tube components, and the relative conditions on the first and last component are fixed.  So the analysis proceeds as in the proof of Theorem \ref{HirzebruchGW}.  The vector on the right and left of the expression are determined by the problem, and the geometry of the tube components contributes $exp(t \mathsf{M}_{F_k})$ as before.

\end{proof}

\section{Derivation of formulas for $\proj^2$}
To prove this theorem, we proceed analogously to the proof of Theorem \ref{HirzebruchGW}.  
\subsection{Degeneration}

We will compute $N_{g,d}^{\bullet}$ by applying the degeneration formula to the following degeneration.  Start with the surface $\proj^2$.  Now degenerate to the normal cone of a line $L$ in $\proj^2$.  The resulting surface will be $F_1 \bigcup \proj^2$, with the two components meeting along $L$ in $\proj^2$ and $E$ in $F_1$.  Next, degenerate to the normal cone of a curve in $F_1$ in the class $C$.  The resulting surface will be $F_1 \bigcup F_1 \bigcup \proj^2 $, where the first two surfaces meet as before, and the $F_1$ and $F_1$ meet along $C$ in the first $F_1$, $E$ in the second.  Repeat this step $n-1$ times more to obtain a chain of $n+2$ surfaces, the first surface $S_0$ isomorphic to $\proj^2$, and the remaining surfaces each isomorphic to $F_1$.  This is the degeneration of $\proj^2$ we will use, and we distribute the original $n$ point conditions by placing one on each of the middle $n$ components.

\subsection{Tube components}
The geometry of the middle $n+1$ components is exactly the same as for the degeneration of $F_1$ we considered earlier, and the contribution to the formula will be the same.

\subsection{Cap components}
The cap $S_{n+1}$ for $\proj^2$ is isomorphic to $F_1$ and the geometry of maps to it is exactly the same as in the analysis for $F_1$.  The cap $S_0$ however is different.  It is isomorphic to $\proj^2$, and given the degeneration and distribution of point conditions we have chosen, the moduli space of maps to this cap is empty.  

\subsection{Putting the pieces together}

\begin{proof}[Proof of Theorem \ref{P2}]
The degeneration we've chosen for $\proj^2$ is very similar to the one we used for $F_1$.  They differ only in the cap $S_0$, so the formula for $\proj^2$ is the same as that for $F_1$ except for the vector on the right side, which encodes the contribution of $S_0$.  Because the space of maps to this cap is empty, the vector $w$ in the formula for $F_1$ is replaced by $v_\emptyset$ here.
\end{proof}

\begin{proof}[Proof of Corollary \ref{relP2}]
We modify the previous proof by taking the degeneration of $\proj^2$ to a union of $n = 3d+g-1- (|\mu|+|\nu| - \ell(\nu))$ surfaces isomorphic to $F_1$ and a cap $S_0$ isomorphic to $\proj^2$, and place a point condition in each component isomorphic to $F_1$.  The geometry of the components in the resulting surfaces are the same as before, except the cap $S_{n+1}$ has been removed, and instead we have a prescribed relative condition $\eta_n = \mu[\ppp]+\nu[\uuu]$ on the last component $S_n$.  This results in the replacement of the vector $v$ in the formula for the generating function of Gromov--Witten invariants with the vector $| \mu,\nu \rangle$.
\end{proof}

\section{Adjoint presentation of formulas}

Note that with the inner product (\ref{inner_prod}),

\begin{eqnarray*}
\alpha_i[\ppp]^{\dagger} = u \alpha_{-i}[\ppp], \hspace{1in} \alpha_{-i}[\ppp]^{\dagger} = u^{-1} \alpha_{i}[\ppp]\\
\alpha_i[\uuu]^{\dagger} = u \alpha_{-i}[\uuu], \hspace{1in} \alpha_{-i}[\uuu]^{\dagger} = u^{-1} \alpha_{i}[\uuu]\\
\end{eqnarray*}

Hence 
\begin{eqnarray*}
\alpha_\nu[\ppp]^{\dagger} = u^{\ell(\nu)} \alpha_{-\nu}[\ppp], \hspace{1in} \alpha_{-\nu}[\ppp]^{\dagger} = u^{-\ell(\nu)} \alpha_{\nu}[\ppp]\\
\alpha_\mu[\uuu]^{\dagger} = u^{\ell(\mu)} \alpha_{-\mu}[\uuu] , \hspace{1in}  \alpha_{-\mu}[\uuu]^{\dagger} = u^{-\ell(\mu)} \alpha_{\mu}[\uuu]\\
\end{eqnarray*}

and

$$\MM_{F_k}(u,Q_1, Q_2)^{\dagger} = \sum_{m>0} \alpha_{-m}[\ppp] \alpha_m[\ppp]
+ Q_1^k Q_2  \sum_{|\nu|-k =|\mu| \geq 0} u^{{\ell(\mu)}-1} \alpha_{-\mu}[\uuu]\alpha_\nu[\uuu]\ \ .$$
The second sum is over all pairs of partitions $\mu$ and $\nu$ whose size differ by $k$.

We can equivalently write the formulas in this paper in terms of the adjoint $\MM_{F_k}(u,Q_1, Q_2)^{\dagger}$.  For example, 

\begin{Corollary} \label{AdjHirzebruchGW}
$\mathsf{Z^{F_k}} (u, Q_1, Q_2,  t) = \big\langle\ w_k \ | \  \exp\big(t\mathsf{M}_{F_k}(u,Q_1, Q_2)^{\dagger} \big) \ | \ v \ \big\rangle.$
\end{Corollary}

\begin{Corollary} \label{AdjP2}
$\mathsf{Z^{\proj^2}} (u, Q,  t) =\big\langle\ v_{\emptyset}\ | \ \exp\big(t\mathsf{M}_{F_1}(u,Q_1, Q_2)^{\dagger}\big) \ | v \ \big\rangle.$
\end{Corollary}

We could also have obtained these formulas by starting with the degeneration of Figure \ref{DegenerationPic} in the opposite order and analyzing the geometries of the components as we did for the degeneration we studied.

\section{Expression of $M_{F_k}$ in terms of fields}

Our operators can also be written in a form which highlights their relationship to vertex operators.  We define the following power series of operators (sometimes called fields).

\begin{align*}
A_+(z) := \sum_{i>0} \alpha_i[p] z^i, \qquad A_-(z) := \sum_{i<0} \alpha_i[p] z^i
\end{align*}

$$A(z, u) := A_+(z) + u A_-(z) $$

\begin{align*}
B_+(z, Q_1) := \sum_{i>0} \alpha_i[1] Q_1^{-i} z^i , \qquad B_-(z, Q_1) := \sum_{i<0} \alpha_i[1] Q_1^{-i} z^i
\end{align*}

$$B(z, u, Q_1) := B_+(z) + u B_-(z)$$

Then 
$$\mathsf{M}_{F_k}(u,Q_1, Q_2) :=    \frac{1}{2} u^{-1} A^2(z, u) \bigg|_{z^0} + Q_2 u^{-1} exp(B(z, u, Q_1)) \bigg|_{z^{-k}}  $$

\section{Relation to other work}

\subsection{Caporaso--Harris}
The formula given in Corollary \ref{AdjP2} is closely related to the recursion discovered by Caporaso and Harris.  Let's compare the calculation of $N_{d,g}^{\bullet}$ using Corollary \ref{AdjP2} and using Caporaso--Harris' recursion.  

The calculation of $N_{d,g}^{\bullet}$ (rather than the entire generating function) in Theorem 2 is obtained by taking the single term 

\begin{equation}
\big\langle\ v_{\emptyset}\ | (\mathsf{M}_{F_1}(u,Q_1, Q_2)^{\dagger})^n \ | 1^d, \emptyset \ \big\rangle.
\end{equation}

This term was obtained by applying the degeneration formula to the degeneration of $\proj^2$ into $\proj^2$ union $n$ copies of $F_1$, as in Figure \ref{DegenerationPic}, where the cap $S_{n+1}$ is isomorphic to $\proj^2$ and meets the rest of the degenerate surface along a line $L$.  

In Caporaso--Harris' calculation, they consider degree $d$ plane curves and their degenerations as the point conditions are moved onto this line $L$.  Their notation for beginning with the trivial relative condition on this line is $(d, \delta, 0, d)$, where $\delta = {{d-1}\choose {2}} - g$.  

Starting with the trivial relative condition on $S_{0}$, which corresponds to the vector $v$ on the right in Theorem \ref{AdjP2}, we apply the operator $\MM^{\dagger}_{F_1}$ $n$ times.  Via the degeneration formula this corresponds geometrically to producing all possible combinatorial types describing broken curves mapping to the very degenerate surface, assigning a multiplicity to each which is the product of numbers associated to each component, and summing the multiplicities over all possible combinatorial types.

One can also consider the recursion of Caporaso--Harris as generating all combinatorial types and associating to each a multiplicity which is the product of the number produced at each step of the recursion.  There is a bijection between our combinatorial types and those of Caporaso--Harris, and they contribute the same multiplicity to the sum.  

There is a difference in how these multiplicities are computed in the two approachs due to different treatments of the labeling of the partitions that denote the relative conditions.  However this difference in combinatorial factors cancels when the contributions from all components are multiplied.  The two summands of $M_{F_1}$ correspond to the two terms in Caporaso--Harris' recursion, namely $\alpha_{-k}[p] \alpha_{k}[p]$ corresponds to the $k$ in their formula, and $\alpha_{-\mu}[\uuu]\alpha_\nu[\uuu]$ corresponds to $I^{\beta' - \beta}$.  

\subsection{Differential Equation of Getzler and Vakil}

We can give another expression of the result of Theorem \ref{relHirzebruchGW}.  Consider $\alpha_{-i}[\uuu]$ and $\alpha_{-i}[\ppp]$ as formal variables.  Then $Y_C$ defined below can be interpreted as the generating function for relative Gromov--Witten invariants with non-trivial relative conditions imposed only along the divisor $C$.

\begin{align*}
Y_C(u,Q_1,Q_2,t,\balpha[\uuu],\balpha[\ppp]) 
&= \exp\big(t\mathsf{M}_{F_k}(u,Q_1, Q_2)\big) \ | \ w_0 \ \big\rangle \\
&= \sum_{g\in \mathbb{Z}} u^{g-1} \sum_{(d_1,d_2)} \  N_{g,(d_1,d_2)}^{\bullet} (\mu_C, \nu_C) \frac{t^n}{n!} \ Q_1^{d_1} Q_2^{d_2} \alpha^{\mu_C}[\uuu] \alpha^{\nu_C}[\ppp],
\end{align*}
where the second equality is implied by Theorem \ref{relHirzebruchGW}.

From the form of the expression defining $Y_C$, it is clear that $Y_C$ satisfies the following differential equation:

\begin{Lemma}\label{CoopDiffC}
$$\frac{\partial}{\partial t} Y_C (u, Q_1, Q_2,  t, \balpha[\uuu],\balpha[\ppp]) = \mathsf{M}_{F_k}(u,Q_1, Q_2)  Y_C (u, Q_1, Q_2,  t, \balpha[\uuu],\balpha[\ppp]).$$
\end{Lemma}

\begin{proof}
$$\frac{\partial}{\partial t} \exp\big(t\mathsf{M}_{F_k}(u,Q_1, Q_2)\big) \ | \ w_0 \ \big\rangle = \mathsf{M}_{F_k}(u,Q_1, Q_2)  \exp\big(t\mathsf{M}_{F_k}(u,Q_1, Q_2)\big) \ | \ w_0\ \big\rangle .$$
\end{proof}

Similarly, we can define 
\begin{align*}
Y_E(u,Q_1,Q_2,t,\balpha[\uuu],\balpha[\ppp]) 
&= \exp\big(t\mathsf{M'}_{F_k}(u,Q_1, Q_2)^{\dagger}\big) \ | \ w_0 \ \big\rangle \\
&= \sum_{g\in \mathbb{Z}} u^{g-1} \sum_{(d_1,d_2)} \  N_{g,(d_1,d_2)}^{\bullet} (\mu_E, \nu_E) \frac{t^n}{n!} \ Q_1^{d_1} Q_2^{d_2} \alpha^{\mu_E}[\uuu] \alpha^{\nu_E}[\ppp] 
\end{align*}

where 
$$\mathsf{M'}_{F_k}(u,Q_1, Q_2) = \sum_{m>0} \alpha_{-m}[\ppp] \alpha_m[\ppp]
+  Q_2  \sum_{|\mu|-n =|\nu| \geq 0} u^{({\ell(\mu)}-1)} \alpha_{-\mu}[\uuu]\alpha_\nu[\uuu]\ \ .$$ 

The second equality is implied by Theorem \ref{relHirzebruchGW} after the formula is expressed in its adjoint form.

From the form of the expression defining $Y_E$, it is clear that $Y_E$ satisfies the following differential equation:

\begin{Lemma}\label{CoopDiffE}
$$
\frac{\partial}{\partial t} Y_E (u, Q_1, Q_2,  t, \balpha[\uuu],\balpha[\ppp]) = \mathsf{M'}_{F_k}(u,Q_1, Q_2)^\dagger  Y_E (u, Q_1, Q_2,  t, \balpha[\uuu],\balpha[\ppp])
$$
\end{Lemma}

\begin{proof}
$$
\frac{\partial}{\partial t} \exp\big(t\mathsf{M'}_{F_k}(u,Q_1, Q_2)^{\dagger}\big) \ | \ w_0 \ \big\rangle  = \mathsf{M'}_{F_k}(u,Q_1, Q_2)^\dagger  \exp\big(t\mathsf{M'}_{F_k}(u,Q_1, Q_2)^{\dagger}\big) \ | \ w_0 \ \big\rangle.$$

\end{proof}
If we rewrite Lemma \ref{CoopDiffE} in terms of the familiar commutation relation $[\frac{d}{dx},x] = 1$, we recover differential equations satisfied by Severi degrees stated by Getzler and Vakil in a different setting.

The algebra $\H [ \proj^1]$  is isomorphic to the algebra generated by $x_i, y_i, i \frac{\partial}{\partial x_i}, i \frac{\partial}{\partial y_i}$ with the usual commutation relations. 
The isomorphism is given by

\begin{align*}
\alpha_{i}[\ppp] \rightarrow i \frac{\partial}{\partial x_i}, \hspace{1in} \alpha_{-i}[\uuu] \rightarrow x_i\\
\alpha_i[\uuu] \rightarrow i \frac{\partial}{\partial y_i}, \hspace{1in} \alpha_{-i}[\ppp] \rightarrow y_i\\
\end{align*}

The commutation relations match: 

\begin{align*}
[\alpha_{i}[\ppp],\alpha_{-i}[\uuu]] = i, \hspace{1in} [\alpha_i[\uuu], \alpha_{-i}[\ppp] ] =  i\\
\end{align*}

becomes

\begin{align*}
[i \frac{\partial}{\partial x_i}, x_i] = i, \hspace{1in} [i \frac{\partial}{\partial y_i}, y_i] = i.
\end{align*}

Under this isomorphism, $Y_E$ becomes 

$$Y_E (u, Q_1, Q_2,  t, \bx,\by) = 1+\sum_{g\in \mathbb{Z}} u^{g-1} \sum_{(d_1,d_2, \mu, \nu)} \  N_{g,(d_1,d_2)}^{\bullet}(\mu_E, \nu_E) \frac{t^n}{n!} Q_1^{d_1} Q_2^{d_2} \bx^{\bmu_E} \by^{\bnu_E},$$

and $\mathsf{M'}_{F_k}(u,Q_1, Q_2)^\dagger$ becomes

$$ \sum_{i} i y_i \frac{\partial}{\partial x_i} + Q_2 u^{-1} exp\left(\sum_{i>0} i u z^i \frac{\partial}{\partial y_i}+  z^{-i} x_i \right) \bigg|_{z^{-k}}.$$

In \cite{G} Getzler reinterprets Caporaso--Harris recursion for $\proj^2$ as a differential equation satisfied by a generating function of all counts of plane curves relative to a line.  In \cite{V} Vakil does the same for Hirzebruch surfaces.  We write here the notation and formula Vakil gives.

First, he defines the generating function
$$G = \sum_{D,g,\alpha,\beta} N^{D,g}(\alpha,\beta) v^D w^{g-1} \left(\frac{x^\alpha}{\alpha !} \right) y^\beta \left(\frac{z^\Gamma}{\Gamma !} \right)$$
where $N^{D,g}(\alpha, \beta)$ is the Severi degree of genus $g$ curves in the class $D$ on a rational surface $X$ with relative condition $\alpha[\uuu] +\beta[\ppp]$ along the divisor $E$.

The function $G$ satisfies the differential equation 
\begin{equation} \label{VakilGetzler}
\frac{ \partial G}{\partial z} = \left( \sum i y_i \frac{\partial}{\partial x_i} + \frac{v^E}{w} \hspace {.1in} exp\left(\sum i w t^i \frac{\partial}{\partial y_i}+ t^{-i}x_i  \right) \bigg|_{z^{-k}} \right) G
\end{equation}

Due to a difference in conventions labeling the relative conditions, 
$$
N_{g,(d_1,d_2)}^{\bullet}(\mu_E, \nu_E) = \frac{N^{D,g}(\alpha,\beta)}{\alpha!}.
$$

With that taken into account, the differential equation (\ref{VakilGetzler}) matches (\ref{CoopDiffE}) once the variables are matched (e.g. $z \rightarrow t$, $\Gamma \rightarrow n$, etc.).

\subsection{Abramovich-Bertram} \label{AbramBert}

Let $N^g_{F_k} (aC + bF)$ be the number of irreducible genus $g$ curves in class $aS + bF$, $a,b \geq 0$, through the appropriate number of points.

The following relation was proved in genus 0 by Abramovich-Bertram and Graber \cite{AB}, \cite{Gr} and later extended to all genera by Vakil in \cite{V}.

\begin{Theorem*}[Vakil, \cite{V}] \label{ABV}
For all $g, a, b > 0$, 
$$N^g_{F_0} (aC + (a + b)F) = \sum_{i=0}^{a-1} {{b + 2i} \choose{i}} N^g_{F_2} (aC + bF-iE).$$
\end{Theorem*}

This Theorem can be recovered by comparing the formula for Gromov--Witten invariants on $F_2$ given by Theorem \ref{HirzebruchGW} to the formula for transverse Gromov--Witten invariants on $F_2$ given by Corollary \ref{logHirzebruchGW}.

\begin{Corollary} \label{ABVii}
For all $g, d_1,d_2$ satisfying $d = d_1-2d_2 \geq 0$, we have the following relationship between the disconnected Gromov--Witten and transverse Gromov--Witten invariants of $F_2$. 
$$N^\bullet_{F_2, g, (d_1, d_2)} = \sum_{i=0}^{d_2} {{d + 2i} \choose{i}} \hat{N}^\bullet_{F_2, g, (d_1, d_2-i)}.$$
\end{Corollary}

\begin{proof}
We examine the formula for Gromov--Witten invariants of $F_2$.  By Theorem \ref{HirzebruchGW},

$$\mathsf{Z^{F_2}} (u, Q_1, Q_2,  t) = \big\langle\ v \ | \  \exp\big(t\mathsf{M}_{F_2}(u,Q_1, Q_2)\big) \ | \ w_2\ \big\rangle$$
where $w_2 = exp(Q_1 \alpha_{-1}[\uuu] + Q_1 Q_2 \alpha_{-1}[\uuu])$.

By Theorem \ref{HirzebruchGW}, we can compute $N^\bullet_{F_2, g, (d_1, d_2)}$ as 

\begin{equation}\label{F2term}
N^\bullet_{F_2, g, (d_1, d_2)} u^{g-1} Q_1^{d_1} Q_2^{d_2} \frac{t^n}{n!} = \langle \frac{\alpha_{-1}^{d_1}[\uuu]}{d_1!} v_\emptyset | \frac{(t M_{F_2})^n}{n!} | \sum_{i=0}^{d_2} \frac{(Q_2 \alpha_{-1} [\uuu])^i}{i!}\cdot \frac{\alpha_{-1}[\uuu] ^{d+i}}{(d+i)!} v_\emptyset \rangle.
\end{equation}

The following numbers are equal:
$$
 \langle \frac{\alpha_{-1}^{d_1}[\uuu]}{d_1!} v_\emptyset | \frac{( M_{F_2})^{n}}{n!} |\frac{( \alpha_{-1} [\uuu])^i}{i!}\cdot \frac{\alpha_{-1}[\uuu] ^{d+i}}{(d+i)!} v_\emptyset \rangle
 = {{d+2i}\choose{i}} \langle \frac{\alpha_{-1}^{d_1}[\uuu]}{d_1!} v_\emptyset | \frac{( M_{F_2})^{n}}{n!} |\frac{\alpha_{-1}[\uuu] ^{d+2i}}{(d+2i)!} v_\emptyset \rangle.
$$
The corollary follows.

\end{proof}

$F_0$ and $F_2$ are deformation equivalent, so $N^\bullet_{F_2, g, (d_1, d_2)} = N^\bullet_{F_0, g, (d_1+d_2, d_2)}$.  As discussed earlier, the transverse Gromov--Witten invariants of $F_2$ are enumerative.  So Corollary \ref{ABVii} recovers Theorem \ref{ABV}, after noting that the summand in Corollary \ref{ABVii} has one more term, because the Corollary concerns disconnected curves while the Theorem is for connected curves.

In \cite{BM} Brugalle and Markwig study the relationship between the enumerative curve counts on $F_k$ and $F_{k+2}$, and in particular also recover the disconnected form of Abramovich and Bertram's relationship between the counts on $F_0$ and $F_2$.

\vspace{+8 pt}
\noindent
Department of Mathematics\\
Harvard University\\
yaim@math.harvard.edu

\end{document}